\documentclass[lettersize,journal]{IEEEtran}
\usepackage{amsmath,amsfonts}
\usepackage{algorithmic}
\usepackage{algorithm}
\usepackage{array}
\usepackage[caption=false,font=normalsize,labelfont=sf,textfont=sf]{subfig}
\usepackage{textcomp}
\usepackage{stfloats}
\usepackage{url}
\usepackage{verbatim}
\usepackage{graphicx}
\usepackage{cite}
\usepackage{xcolor}
\usepackage{color}
\usepackage{multirow}
\usepackage{enumitem}
\usepackage{threeparttable}
\usepackage{amsmath,amsthm}
 \usepackage{amssymb}
\usepackage{amsfonts}
\usepackage{framed}
\usepackage{lipsum}
\usepackage{graphicx}
\usepackage{mathrsfs}
\usepackage{mathtools}
\usepackage{booktabs}

\newtheorem{theorem}{Theorem}

\newtheorem{proposition}[theorem]{Proposition}

\allowdisplaybreaks[4]

\begin{document}

\title{Distributionally Robust Frequency-Constrained Microgrid Scheduling Towards Seamless Islanding}

\author{Lun Yang, Haoxiang Yang, Xiaoyu Cao, Xiaohong Guan,~\IEEEmembership{Fellow,~IEEE}
}


\IEEEpubid{}

\maketitle

\begin{abstract}
Unscheduled islanding events of microgrids result in the transition between grid-connected and islanded modes and induce a sudden and unknown power imbalance, posing a threat to frequency security. To achieve seamless islanding, we propose a distributionally robust frequency-constrained microgrid scheduling model considering unscheduled islanding events. This model co-optimizes unit commitments, power dispatch, upward/downward primary frequency response reserves, virtual inertia provisions from renewable energy sources (RESs), deloading ratios of RESs, and battery operations, while ensuring the system frequency security during unscheduled islanding. We establish an affine relationship between the actual power exchange and RES uncertainty in grid-connected mode, describe RES uncertainty with a Wasserstein-metric ambiguity set, and formulate frequency constraints under uncertain post-islanding power imbalance as distributionally robust quadratic chance constraints, which are further transformed by a tight conic relaxation. We solve the proposed mixed-integer convex program and demonstrate its effectiveness through case studies.

\end{abstract}


\vspace{-0.3cm}

\section{Introduction}
\IEEEPARstart{M}{icrogrids} integrate diesel generators (DGs), renewable energy sources (RESs), battery energy storage systems (BESSs), and local demands and provide a promising pathway to achieve flexible and low-carbon power supply \cite{IEEE-2018, Hirsch-RSER-2018}. Microgrids can operate in both grid-connected and islanded modes~\cite{Katiraei-TPD-2005}. Their islanding capability is critical to enhance power supply reliability under extreme conditions or main grid faults. Nevertheless, sudden islanding leads to a large power imbalance, which makes maintaining frequency stability difficult since microgrids alone cannot provide enough rotational inertia~\cite{Zheng-IJEPES-2018}. Therefore, developing a new microgrid scheduling approach is necessary to guarantee frequency security seamlessly through unscheduled islanding transitions.

Several previous works have studied the problem regarding microgrid operations with islanding transitions, primarily focused on pre-positioning reserves. Liu \emph{et al.}~\cite{Liu-EPSR-2017} propose the probability of successful islanding (PSI) concept and include it in microgrid scheduling, where upward/downward reserves are optimized. Hemmati~\emph{et al.}~\cite{Hemmati-IEEESJ-2020} introduce the PSI to the reconfigurable microgrid scheduling. Wu~\emph{et al.}~\cite{Wu-TSG-2020} propose a two-stage chance-constrained model to determine reserves for the microgrid operation mode switching. However, such a strategy without considering frequency dynamics may be insufficient for frequency security during the islanding transition.

To ensure frequency security in case of a contingency, a natural idea is to incorporate post-contingency dynamic frequency requirements, i.e., constraints on rate-of-change-of-frequency (RoCoF), maximum frequency deviation (or frequency nadir/zenith), and quasi-steady state, into system scheduling. Such \emph{frequency-constrained scheduling problems} have been extensively studied in bulk power systems \cite{Teng-TPS-2016, Yang-TSG-2022, Wen-TPS-2016, Badesa-TPS-2020, Malley-TPS-2022, Yang-TSE-2023}, but for microgrid scheduling following an unscheduled islanding event, research is still limited. 

Wen \emph{et al.}~\cite{Wen-TPS-2019} embed frequency constraints, derived from discretized-time frequency dynamics, in a microgrid dispatch model considering unscheduled islanding events.
Javadi \emph{et al.}~\cite{Javadi-TPS-2022} expand upon this formulation, adding the commitment decision of generators. However, neither of those takes into account the frequency support from RESs. In fact, inverter-based resources (IBRs) like RESs are able to provide frequency support; e.g., via deloading~\cite{Dreidy-RSER-2017}. It may be infeasible to obtain a microgrid scheduling scheme with frequency security only considering frequency support provided by DGs. Gholami \emph{et al} \cite{Gholami-TCNS-2019} consider the frequency support from IBRs and propose a comprehensive frequency-constrained real-time operation framework for multi-microgrids under an islanding event, but the ON/OFF schedule of generating units is predetermined, and the system inertia is fixed \emph{a priori}. 

Subsequently, Zhang \emph{et al.}~\cite{Zhang-TPS-2021} propose a deep learning-assisted frequency-constrained microgrid scheduling (FCMS) model considering frequency support from wind turbines via deloading. Chu \emph{et al.}~\cite{Chu-TSG-2021} co-optimize the non-critical load shedding and the virtual inertia from wind turbines. However, both works neglect the uncertainty associated with wind power outputs. Moreover, the power imbalance caused by islanding is assumed to have a prespecified direction (i.e., importing power) in \cite{Zhang-TPS-2021} and roughly replaced with a \emph{known} load increase in \cite{Chu-TSG-2021} to facilitate the formulation of frequency nadir constraints. Such settings may not be true in practice as the post-islanding power imbalance is the power exchange between the microgrid and the main gird in grid-connected mode, which is unknown \emph{a priori} and determined by optimal grid-connected microgrid scheduling. The recent work of~\cite{Qi-TSG-2023} develops a stochastic FCMS model considering a random load increase but within the context of an isolated microgrid. 

Current works regarding FCMS under unscheduled islanding have the following gaps. Firstly, none of the existing studies account for the \emph{uncertain} power imbalance, which results from the RES uncertainty. It is unclear how such an uncertain power imbalance, together with the frequency constraints, will impact the scheduling decisions. Secondly, existing works lack a holistic co-optimization of generators' scheduling, primary frequency response (PFR) reserves, virtual inertia of BESSs and RESs, as well as RES deloading ratios to achieve frequency security. In this work, we aim to fill the research gaps above with the following contributions:

\begin{itemize}
    \item  We propose a distributionally robust FCMS (DR-FCMS) model, co-optimizing generators' scheduling, PFR reserves, virtual inertia of BESSs/RESs, and RES deloading ratios, under sudden islanding with RES uncertainty.  
    \item We model RES uncertainty via a new Wasserstein-metric ambiguity set and formulate the frequency constraints under uncertain post-islanding power imbalance as distributionally robust (DR) chance constraints, where we design a tight conic reformulation and solve the proposed DR-FCMS model as a mixed-integer convex program.
    \item We perform extensive benchmark analyses against alternative decision models and uncertainty models, with which we show our method yields an economical microgrid schedule with frequency security. 
\end{itemize}
The rest of the paper is organized as follows: Section~\ref{sec.Frequency dynamics} formulates the frequency constraints; Section~\ref{sec.DR-FCMS} connects the frequency constraints to the microgrid scheduling and describes the uncertainty model; Section~\ref{sec.Methods} covers the tractable reformulation of the DR chance constraints (DRCC); Section~\ref{sec.Case} details the benchmark analyses; and Section~\ref{sec.Conclusion} concludes the paper with a brief discussion of future work. 

\section{Formulation of Frequency Constraints}\label{sec.Frequency dynamics}
Unscheduled islanding leads to a sudden loss of power exchange between the microgrid and the main grid, which affects microgrid's frequency stability. The frequency decreases if the microgrid imports power from the main grid before islanding and increases vice versa, as illustrated in Fig.~\ref{fig1:frequency_dynamics}. In a microgrid, besides synchronous DGs, the frequency support from IBRs, such as RESs and BESSs, is useful to contain the frequency excursion. The RESs operate in deloading mode for upward PFR reserve to arrest the frequency drop while BESSs can contain the frequency drop/rise by discharging/charging.  We approximate frequency dynamics of the microgrid by the swing equation \cite{Kundur-Book-1994}, assuming we use virtual inertia control for RESs and droop control for BESSs~\cite{Shen-TSE-2023}:
\begin{figure}[t]
    \centering   \includegraphics[width=0.45\textwidth]{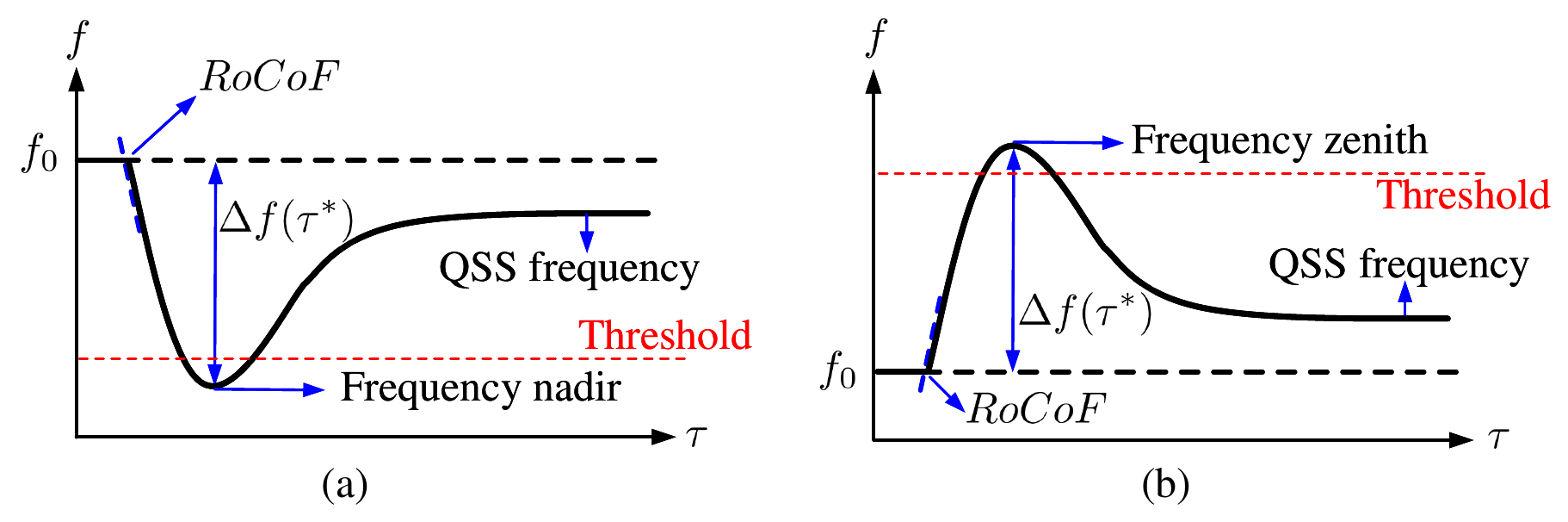}
        \vspace{-0.4cm}
	\caption{Microgrid frequency excursion after an unscheduled islanding event: (a) frequency falls when $p_t^{imb}>0$ or (b) frequency rises when $p_t^{imb}<0$.}
    \label{fig1:frequency_dynamics}
\end{figure}

{\small \begin{equation}
\label{SWINGequations}
 2H_t\frac{d\Delta f(\tau)}{d\tau}=\sum_{i \in \mathcal{G}} \Delta P_i(\tau)+\sum_{s \in \mathcal{S}} \Delta P_s(\tau) +\sum_{b \in \mathcal{B}} \Delta P_b(\tau)-p_t^{imb}
\end{equation}}

where $\Delta f(\tau)$ is the frequency deviation caused by the post-islanding power imbalance $p_t^{imb}$. We index the DGs, RESs and BESSs by \(i \in \mathcal{G}\), \(s \in \mathcal{S}\), and \(b \in \mathcal{B}\), respectively. A positive \(p_t^{imb}\) indicates that the microgrid imports power from the main grid. We denote the PFRs provided by DGs, RESs, and BESSs as $\Delta P_i(\tau)$, $\Delta P_s(\tau)$, and $\Delta P_b(\tau)$. Note that the load damping level in the power-electronic dominated microgrid is generally low and thus neglected in (\ref{SWINGequations}) similar to the setting in \cite{Malley-TPS-2022,Badesa-TPS-2020}.

The system aggregated inertia $H_t$ in (\ref{SWINGequations}) contributed from DGs, RESs, and BESSs is calculated by

{\small
\begin{equation*} 
H_t = \frac{1}{f_0} \left(\sum_{i\in \mathcal{G}}H_i^G g_i^{P,\max} x_{i,t}+\sum_{s\in \mathcal{S}}H_{s,t}^S p_s^{\max} + \sum_{b\in \mathcal{B}}H_{b,t}^B p_b^{\max}\right)
\end{equation*}
}

where $f_0$ is the nominal frequency, $g_i^{P,\max}$, $p_s^{\max}$, and $p_b^{\max}$ denote power output capacities. The parameter \(H_i^G\) is the inertia constant of DG $i$. The binary variable $x_{i,t}$ determines if DG $i$ is on at time $t$. The variables $H_{s,t}^S$ and $H_{b,t}^B$ are virtual inertia of RES $s$ and BESS $b$, which will be optimized in the microgrid scheduling.

In accordance with \cite{Teng-TPS-2016, Yang-TSG-2022, Wen-TPS-2016, Badesa-TPS-2020, Malley-TPS-2022, Yang-TSE-2023}, a linear ramping model is adopted to approximate the PFR dynamics, whose effectiveness is demonstrated in \cite{Badesa-TPS-2020, Yang-TSE-2023}. Indeed, detailed PFR dynamics modeling can enhance the accuracy but hinders deriving the closed-form solution to (\ref{SWINGequations}) and prohibits the tractability of constraints on frequency requirements. We present the PFRs $\Delta P_i(\tau)$, $\Delta P_s(\tau)$, and $\Delta P_b(\tau)$ with a linear ramping model:

{\small
\begin{subequations}
\begin{align}
\label{PFRgen}
&\Delta P_i(\tau)=\left\{\begin{array}{*{20}{l}}
0,&\textrm{if}~0\leq\tau<T_{DB}\vspace{1ex}\\
R_{i,t}^{G}\frac{\tau-T_{DB}}{T_G},&\textrm{if}~T_{DB}\leq\tau<T_G+T_{DB}\vspace{1ex}\\
R_{i,t}^{G}, &\textrm{if}~\tau \geq T_{G}+T_{DB}
\end{array}\right.\\
\label{PFRres}
&\Delta P_s(\tau)=\left\{\begin{array}{*{20}{l}}
R_{s,t}^{S}\frac{\tau}{T_E}, &\textrm{if}~0\leq\tau<T_E \vspace{1ex}\\
R_{s,t}^{S},&\textrm{if}~\tau \geq T_E
\end{array}\right.\\
\label{PFRess}
&\Delta P_b(\tau)=\left\{\begin{array}{*{20}{l}}
R_{b,t}^{B}\frac{\tau}{T_E},~~\textrm{if}~~0\leq\tau<T_E\vspace{1ex}\\
R_{b,t}^{B},~~~~~~\textrm{if}~~ \tau \geq T_E
\end{array}\right.
\end{align}

\end{subequations}}where we denote PFR reserve capacities as $R_{i,t}^{G}$, $R_{s,t}^{S}$, and $R_{b,t}^{B}$, respectively. $T_{DB}$ is the delivery delay of DG due to the deadband of DG's governor. Since governors control the delivery speed of PFRs from DGs while that of PFRs from IBRs, including RESs and BESSs, are controlled by power electronic devices, we have the relationship between DGs' and IBRs' delivery times as $T_G>T_E$.

Given the frequency dynamics above, we formulate the following frequency constraints, including limitations on RoCoF, maximum frequency deviation, and quasi-steady state.

\subsubsection{Constraints on RoCoF}
The greatest magnitude of RoCoF occurs at $\tau=0$ after the islanding event. From~(\ref{SWINGequations}), we have $RoCoF=\frac{-p_t^{imb}}{2H_t}$. As the sign of \(p_t^{imb}\) is unknown, we model two-sided constraints on RoCoF: 
\begin{align}
    \label{RoCoF}
    -2H_t \cdot RoCoF^{\max}\leq p_t^{imb}\leq 2H_t\cdot RoCoF^{\max}
\end{align}
where $RoCoF^{\max}$ is the permissible maximum RoCoF.
\subsubsection{Constraints on Maximum Frequency Deviation} The maximum frequency deviation occurs at frequency nadir or zenith when PFRs fully cover the power imbalance. Equations~(\ref{PFRgen})-(\ref{PFRess}) show the total PFRs is a piecewise linear function. The maximum frequency deviation must occur before $T_G+T_{DB}$ so that the quasi-steady state constraint (\ref{QSS}) can hold. 
In practice, PFR controlled by power electronic devices is fast ($T_E\approx 1$ s). Only when the system has extremely low inertia can the maximum frequency deviation occur within $T_E$, but the constraint on RoCoF is hard to hold in this situation. Hence, it is reasonable to assume that the maximum frequency deviation occurs after $T_E$ but before $T_G+T_{DB}$. 
%
%
%
%
It is shown in \cite{Badesa-TPS-2020} that the maximum frequency deviation is:

{\small
\begin{align}
\label{maximumfrequencydeviation}
\Delta f(\tau^*)=\frac{1}{2H_t}\left(\frac{-\left(p_t^{imb}-R^E_t+R_t^GT_{DB}/T_G\right)^2}{2R^G_t/T_G}-\frac{R^E_tT_E}{2}\right).
\end{align}}

The sign of maximum frequency deviation $\Delta f(\tau^*)$ depends on the sign of $p_t^{imb}$. Therefore, we construct the constraint on $\Delta f(\tau^*)$ for two cases: $p_t^{imb}>0$ and $p_t^{imb}<0$.

1) When $p_t^{imb}>0$, the system frequency drops and $\Delta f(\tau^*)<0$, requiring upward PFR reserves, $R^G_t\geq 0$ and $R^E_t\geq 0$. We let $R^G_t=R^{G,U}_t$, $R^E_t=R^{E,U}_t$ where $R^{G,U}_t=\sum_{i\in \mathcal{G}} R_{i,t}^{G,U}$ and  $R^{E,U}_t=\left(\sum_{s\in \mathcal{S}}R_{s,t}^{S,U}+\sum_{b\in\mathcal{B}}R_{b,t}^{B,U}\right)$ denote the total upward PFR reserves of DGs and IBRs (i.e., RESs and BESSs). Constraint on the maximum frequency deviation is formulated as a rotated conic constraint:
%
%
\begin{align}
\label{FnadirUP_SOC}
\nonumber \left(p_t^{imb}-R^{E,U}_t+R_t^{G,U}T_{DB}/T_G\right)^2 \\ \leq \left(4\Delta f^{\max}H_t-R^{E,U}_tT_E\right)R^{G,U}_t/T_G
\end{align}
where $\Delta f^{\max}$ is the allowable maximum frequency deviation.

2) When $p_t^{imb} < 0$, the system frequency rises and $\Delta f(\tau^*)>0$, requiring downward PFR reserves, $R^G_t\leq 0$ and $R^E_t\leq 0$. We let $R^G_t=-R^{G,D}_t$, $R^E_t=-R^{E,D}_t$ where $R^{G,D}_t=\sum_{i\in \mathcal{G}} R_{i,t}^{G,D}$ and  $R^{E,D}_t=\sum_{b\in\mathcal{B}}R_{b,t}^{B,D}$ are the total downward PFR reserves of DGs and BESSs. Similar to~\eqref{FnadirUP_SOC}, we formulate a rotated conic constraint:
\begin{align}
\label{FnadirDW_SOC}
\nonumber \left(p_t^{imb}+R^{E,D}_t-R_t^{G,D}T_{DB}/T_G\right)^2 \\ \leq\left(4\Delta f^{\max}H_t-R^{E,D}_tT_E\right)R^{G,D}_t/T_G.
\end{align}

The maximum frequency deviation constraints above, including two cases $p_t^{imb}> 0$ and $p_t^{imb} < 0$, belong to the type of \emph{if-then} conditional constraint. By introducing the big-M, we convert the above \emph{if-then} conditional constraint into:
\begin{subequations}
\label{FnadirBinary_reformulation}
\begin{align}
\label{PowerImbalanceDirection}
&-M(1-u_t)\leq p_t^{imb}\leq Mu_t \\
\label{zt_UP}
& -M(1-u_t)\leq z_t-R^{E,U}_t+\frac{R_t^{G,U}T_{DB}}{T_G}\leq M(1-u_t)\\
\label{zt_DW}
& -Mu_t\leq z_t+R^{E,D}_t-\frac{R_t^{G,D}T_{DB}}{T_G}\leq Mu_t\\
\label{yt_UP}
& -M(1-u_t)\leq y_t-\frac{R^{G,U}_t}{T_G}\leq M(1-u_t)\\
\label{yt_DW}
& -Mu_t\leq y_t-\frac{R^{G,D}_t}{T_G}\leq Mu_t\\
\label{Nonegative_PFR}
& R^{G,U}_t,R^{G,D}_t, R^{E,U}_t,R^{E,D}_t \geq 0\\
\label{xt_UP}
& -M(1-u_t)\leq x_t-\left(4\Delta f^{\max}H_t-R^{E,U}_tT_E\right)\leq M(1-u_t)\\
\label{xt_DW}
& -Mu_t\leq x_t-\left(4\Delta f^{\max}H_t-R^{E,D}_tT_E\right)\leq Mu_t\\
\label{Fnadir_SOC_reformulation}
&\left(p_t^{imb}-z_t\right)^2\leq x_ty_t
\end{align}
\end{subequations}
where we introduce a binary variable $u_t$ to describe the unknown direction of the post-islanding power imbalance, $M$ is a predefined positive big number, and $x_t$,  $y_t$, and $z_t$ are optimization variables introduced to enable the reformulation.

\subsubsection{Constraints on Quasi-Steady State} To stabilize frequency eventually after unscheduled islanding, the amount of total available PFRs is required to cover the power imbalance. Considering the unknown sign of $p_t^{imb}$, we formulate the quasi-steady state constraints as the following two-sided form:
\begin{align}
\label{QSS}
-R^{G,D}_t-R^{E,D}_t\leq p_t^{imb}\leq  R^{G,U}_t+ R^{E,U}_t.
\end{align}

\vspace{-0.2cm}
\section{DR-FCMS Formulation} \label{sec.DR-FCMS}
We first describe how we model RES uncertainty and the affinely adaptive decision rule. We then introduce how we account for uncertainty-related constraints, including frequency constraints, via a DRCC approach. Finally, we formulate the full DR-FCMS problem in the event of unscheduled islanding.

\vspace{-0.3cm}
\subsection{Network Representation}
The microgrid is represented by a graph $(\mathcal{N}, \mathcal{L})$, where $\mathcal{N}$ $\mathcal{L}$ are the sets of nodes and lines. 
We assume that the graph has a tree structure. The root node of the graph is indexed by 1, and $\mathcal{N}^{\dag}$=$\mathcal{N}\setminus\{1\}$ denotes the set of non-root nodes. The sets of DGs, RESs, BESSs, and loads are denoted as $\mathcal{G}$, $\mathcal{S}$, $\mathcal{B}$, and $\mathcal{D}$, respectively. 
We assume that there is one DG, one BESS, one RES, and one load at each \emph{non-root} node, such that $|\mathcal{N}^{\dag}|=|\mathcal{G}|=|\mathcal{S}|=|\mathcal{B}|=|\mathcal{D}|$. Non-root nodes without the DG, BESS, RES, or load can be handled by setting the corresponding parameters to zeros, and nodes with multiple entries can be addressed by a summation.

\vspace{-0.3cm}
\subsection{Modeling of System Responses under Uncertainty}\label{sec:modeling_system_response}

In this section, we first introduce the uncertainty model of active power forecasts. Then we adopt an affinely adaptive policy similar to~\cite{Roald-C-2016} for microgrids' responses to RES uncertainty via DGs, BESSs, power exchanges, and power flows. 

\subsubsection{RES Uncertainty}
The RES active power output is considered uncertain as forecasts contain random errors. Therefore, we model the available active power outputs of RESs by the sum of a forecast value and a random forecast error:
\begin{align}
\label{uncer_Renew}
\tilde{p}_{s,t}=p_{s,t}^{f}+\xi_{s,t}
\end{align}
where $\tilde{p}_{s,t}$, $p_{s,t}^f$, and $\xi_{s,t}$ are the available active power output, forecast value, and forecast error of RES.

RESs operate in the deloading mode to provide frequency support under an islanding event. The available power output is split into an actual power output, \(\tilde{p}_{s,t}^c\), and a portion to provide reserves for the virtual inertia response and PFR, \(\tilde{r}_{s,t}^c\)~\cite{Dreidy-RSER-2017}. We optimize the deloading ratio $\delta_{s,t}$:
\begin{subequations}
\begin{align}
\label{ReneP_curt}
&\tilde{p}_{s,t}^c=(1-\delta_{s,t})\tilde{p}_{s,t}=(1-\delta_{s,t})p_{s,t}^{f}+(1-\delta_{s,t})\xi_{s,t}\\
\label{Rene_injection}
&\tilde{r}_{s,t}^c=\delta_{s,t}\tilde{p}_{s,t}=\delta_{s,t}p_{s,t}^{f}+\delta_{s,t}\xi_{s,t}
\end{align}
\end{subequations} 

We assume a \emph{constant power factor} $\cos\phi$ for uncertain renewable energy injections and $\varphi=\sqrt{(1-\cos^2\phi)/(\cos^2\phi)}$, so we can model the actual reactive power injection by:
\begin{align}
\label{ReneQ_curt}
\tilde{q}_{s,t}^c=\varphi\tilde{p}_{s,t}^c=\varphi(1-\delta_{s,t})p_{s,t}^{f}+\varphi(1-\delta_{s,t})\xi_{s,t}
\end{align}

\subsubsection{DGs} Based on the affine policy, the actual active and reactive power outputs of DG $i$ are modeled by:
\begin{subequations}
\begin{align}
\label{DG_active_output}
&\tilde{g}_{i,t}^P=g_{i,t}^P-(\alpha_{i,t}^P)^{\text T}\xi_t\\
\label{DG_reactive_output}
&\tilde{g}_{i,t}^Q=g_{i,t}^Q-(\alpha_{i,t}^Q)^{\text T}\xi_t
\end {align}
\end{subequations}
where $\tilde{g}_{i,t}^P$/$\tilde{g}_{i,t}^Q$ are the actual active/reactive power outputs in the presence of uncertainty, and $g_{i,t}^P$/$g_{i,t}^Q$ are the nominal active/reactive power outputs in the absence of uncertainty, i.e., $\xi_t=0$. The participation factor vectors $\alpha_{i,t}^P/\alpha_{i,t}^Q\in \mathbb{R}^{|\mathcal{S}|\times 1}_+$ are subject to optimization whose entries $\alpha_{i,s,t}^P$/$\alpha_{i,s,t}^Q$ describe the reactions of DG $i$ to uncertain forecast errors of RES $s$.

\subsubsection{BESSs} To balance the uncertain active power forecast errors of RESs, the BESS is controlled to adjust its charging or discharging power with an affine policy:
\begin{subequations}
\begin{align}
\label{Actual_chargePower}
&\tilde{p}_{b,t}^{ch}=p_{b,t}^{ch}+(\beta_{b,t}^{ch})^{\text T}\xi_t\\
\label{Actual_dischargePower}
&\tilde{p}_{b,t}^{dch}=p_{b,t}^{dch}-(\beta_{b,t}^{dch})^{\text T}\xi_t
\end{align}
\end{subequations}
where $\tilde{p}_{b,t}^{ch}$/$\tilde{p}_{b,t}^{dch}$ are the actual charging/discharge power of BESSs and $p_{b,t}^{ch}$/$p_{b,t}^{dch}$ are the nominal charging/discharging power corresponding to $\xi_t=0$. The participation factor vectors $\beta_{b,t}^{ch}/\beta_{b,t}^{dch}\in \mathbb{R}^{|\mathcal{S}|\times 1}_+$ are subject to optimization. 

\subsubsection{Power Exchanges} We model the actual power exchanges between the microgrid and main grid by:
\begin{subequations}
\begin{align}
\label{Actual_active_output_exchange}
&\tilde{p}_{t}^{ex}=p_{t}^{ex}-(\gamma_{t}^P)^{\text T}\xi_t\\
\label{Actual_reactive_output_exchange}
&\tilde{q}_{t}^{ex}=q_{t}^{ex}-(\gamma_{t}^Q)^{\text T}\xi_t
\end {align}
\end{subequations}
where $\tilde{p}_{t}^{ex}$/$\tilde{q}_{t}^{ex}$ are the actual active/reactive power exchanges corresponding to $\xi_t\neq0$ while the nominal power exchanges $p_{t}^{ex}$/$q_{t}^{ex}$ correspond to $\xi_t=0$. The participant factor vectors $\gamma_{t}^P/\gamma_{t}^Q\in \mathbb{R}^{|\mathcal{S}|\times 1}_+$ are subject to optimization.

\subsubsection{Line Power Flows} The actual active and reactive power flows across line $ij$ are modeled by:
\begin{align}
\label{Actual_active_Lineflow}
\tilde{f}_{ij,t}^{P}=f_{ij,t}^{P}+(\chi_{ij,t}^P)^{\text T}\xi_t\\ 
\label{Actual_reactive_Lineflow}
\tilde{f}_{ij,t}^{Q}=f_{ij,t}^{Q}+(\chi_{ij,t}^Q)^{\text T}\xi_t
\end {align}
where $\tilde{f}_{ij,t}^P$/$\tilde{f}_{ij,t}^Q$ are the actual active/reactive power flows under uncertainty, and $f_{ij,t}^P$/$f_{ij,t}^Q$ are the nominal active/reactive line power flow set points. The participation factor vectors $\chi_{ij,t}^P/\chi_{ij,t}^Q\in \mathbb{R}^{|\mathcal{S}|\times 1}$ are subject to optimization.

\subsubsection{Voltage Magnitudes} The actual squared voltage magnitude is modeled by:
\begin{align}
\label{Actual_voltage}
\tilde{V}_{i,t}=V_{i,t}+\pi_{i,t}^{\text T}\xi_t
\end {align}
where $\tilde{V}_{i,t}/V_{i,t}$ is the actual/nominal squared voltage magnitude and we optimize the participation factor $\pi_{i,t}\in \mathbb{R}^{|\mathcal{S}|\times 1}$.

\vspace{-0.3cm}
\subsection{Objective Function}
The objective function of the grid-connected FCMS under uncertainty is to minimize the total cost:
\begin{align}
\label{obj}
\nonumber & \min \sum_{t\in\mathcal{T}}\sum_{i\in \mathcal{G}}\left(c_i^{SU}z_{i,t}^u+c_i^{SD}z_{i,t}^d+c_i^{NO}x_{i,t}+c_i\sup_{\mathbb{P}\in\mathcal{P}_t} \mathbb{E}_{\mathbb{P}} \left[\tilde{g}_{i,t}^P\right]\right.\\
\nonumber &\left.+c_i^{G,R}(R_{i,t}^{G, U}+R_{i,t}^{G, D})\right)+\sum_{t\in\mathcal{T}}\sum_{s\in\mathcal{S}}\left(c_s^{S,IR}IR_{s,t}^{S,U}+c_s^{S,R}R_{s,t}^{S,U}\right)\\
\nonumber & \sum_{t\in\mathcal{T}}\sum_{b\in\mathcal{B}} \left(c_b^{B}\sup_{\mathbb{P}\in\mathcal{P}_t} \mathbb{E}_{\mathbb{P}}\left [\tilde{p}_{b,t}^{ch}+\tilde{p}_{b,t}^{dch}\right]+c_b^{B,IR}(IR_{b,t}^{B,U}+IR_{b,t}^{B, D})\right.\\
&\left. +c_b^{B,R}(R_{b,t}^{B,U}+R_{b,t}^{B, D})\right)+\sum_{t\in\mathcal{T}}\lambda_t \sup_{\mathbb{P}\in\mathcal{P}_t} \mathbb{E}_{\mathbb{P}}\left[\tilde{p}_{t}^{ex}\right]
\end{align}
where the first five terms represent the start-up, shut-down, no-load, expected generation, and DG PFR reserve costs. The sixth and seventh terms represent inertia response (IR) and RES PFR reserve costs. The eighth to tenth terms are the expected charging/discharging, IR, and BESS PFR reserve costs. The last term represents the expected cost of power exchange with the main grid, where $\lambda_t$ is the unit power exchange price. 
Note that we compute the expected cost against the worst-case distribution of $\xi_t$ within an ambiguity set $\mathcal{P}_t$, which is detailed in Section IV-A.

\vspace{-0.4cm}
\subsection{Microgrid Normal Operation Constraints}
We adopt the \emph{LinDistFlow} model in \cite{Baran-TPD-2016} to describe the microgrid. The operational constraints of microgrid based on the \emph{LinDistFlow} setting are introduced as follows.

\subsubsection{Constraints on Network Power Flows}
We describe the power flow constraints of the microgrid as follows, where we omit $``\forall t\in \mathcal{T}"$ for simplicity:
\begin{subequations}
\begin{align}
\label{Active_Power_Balance_Root_node}
&\tilde{p}_{t}^{ex}=\sum_{j\in \mathcal{L}_i}\tilde{f}_{ij,t}^P, \forall i\in \{1\}\\
\label{Reactive_Power_Balance_Root_node}
&\tilde{q}_{t}^{ex}=\sum_{j\in \mathcal{L}_i}\tilde{f}_{ij,t}^Q, \forall i\in \{1\}\\
\label{Active_Power_Balance_non_Root_node}
&\tilde{g}_{i,t}^P+\tilde{p}_{i,t}^{dch}-\tilde{p}_{i,t}^{ch}+\tilde{p}_{i,t}^c-d_{i,t}^P=\sum_{j\in \mathcal{L}_i}\tilde{f}_{ij,t}^P, \forall i\in \mathcal{N}^{\dag}  \\
\label{Reactive_Power_Balance_non_Root_node}
&\tilde{g}_{i,t}^Q+\tilde{q}_{i,t}^c-d_{i,t}^Q=\sum_{j\in \mathcal{L}_i}\tilde{f}_{ij,t}^Q, \forall i\in \mathcal{N}^{\dag}\\
\label{voltage_difference}
&\tilde{V}_{i,t}=\tilde{V}_{j,t}-2(R_{ij}\tilde{f}_{ij,t}^P+X_{ij}\tilde{f}_{ij,t}^Q), \forall ij\in \mathcal{L}
\end{align}
\end{subequations}
where (\ref{Active_Power_Balance_Root_node}) and (\ref{Reactive_Power_Balance_Root_node}) are the active and reactive power balance equations at the root node while  (\ref{Active_Power_Balance_non_Root_node}) and (\ref{Reactive_Power_Balance_non_Root_node}) are the active and reactive power balance equations at each non-root node. Constraint (\ref{voltage_difference}) represents the voltage drop equation across each line. 


By separating the uncertainty-independent and uncertainty-dependent parts, constraints (\ref{Active_Power_Balance_Root_node})-(\ref{voltage_difference}) are equivalent to enforcing the following constraints:

{\small
\begin{subequations}
\begin{align}
\label{Active_Power_Balance_Root_node_nominal}
&p_{t}^{ex}=\sum_{j\in \mathcal{L}_i}f_{ij,t}^P, \forall i\in \{1\}\\
\label{Reactive_Power_Balance_Root_node_nominal}
&q_{t}^{ex}=\sum_{j\in \mathcal{L}_i}f_{ij,t}^Q, \forall i\in \{1\}\\
\label{Active_Power_Balance_non_Root_node_nominal}
&g_{i,t}^P+p_{i,t}^{dch}-p_{i,t}^{ch}+(1-\delta_{i,t})p_{i,t}^{f}-d_{i,t}^{P}=\sum_{j\in \mathcal{L}_i}f_{ij,t}^P, \forall i\in \mathcal{N}^{\dag}\\
\label{Reactive_Power_Balance_non_Root_node_nominal}
&g_{i,t}^Q+\varphi(1-\delta_{i,t})p_{i,t}^{f}-d_{i,t}^{Q}=\sum_{j\in \mathcal{L}_i}f_{ij,t}^Q, \forall i\in \mathcal{N}^{\dag}\\
\label{voltage_difference_nominal}
&V_{i,t}=V_{j,t}-2(R_{ij}f_{ij,t}^P+X_{ij}f_{ij,t}^Q), \forall ij\in \mathcal{L}\\
\label{Active_Power_Balance_Root_node_real}
&-\gamma_{t}^{P}=\sum_{j\in \mathcal{L}_i}\chi_{ij,t}^P, \forall i\in \{1\}\\
\label{Reactive_Power_Balance_Root_node_real}
&-\gamma_{t}^{Q}=\sum_{j\in \mathcal{L}_i}\chi_{ij,t}^Q, \forall i\in \{1\}\\
\label{Active_Power_Balance_non_Root_node_real}
&-\alpha_{i,t}^P-\beta_{i,t}^{dch}-\beta_{i,t}^{ch}+(1-\delta_{i,t})\Lambda_{i}=\sum_{j\in \mathcal{L}_i}\chi_{ij,t}^P, \forall i\in \mathcal{N}^{\dag}\\
\label{Reactive_Power_Balance_non_Root_node_real}
&-\alpha_{i,t}^Q+\varphi(1-\delta_{i,t})\Lambda_{i}=\sum_{j\in \mathcal{L}_i}\chi_{ij,t}^Q, \forall i\in \mathcal{N}^{\dag}\\
\label{voltage_difference_real}
&\pi_{i,t}=\pi_{j,t}-2(R_{ij}\chi_{ij,t}^P+X_{ij}\chi_{ij,t}^Q), \forall ij\in \mathcal{L}
\end{align}
\end{subequations}}

\noindent where $\Lambda_i \in \mathbb{R}^{|\mathcal{N}^\dag|\times 1}$ is a vector with the $i$-th element being 1 while others being zero. Constraints (\ref{Active_Power_Balance_Root_node_nominal})-(\ref{voltage_difference_nominal}) and (\ref{Active_Power_Balance_Root_node_real})-(\ref{voltage_difference_real}) enforce the power balance and voltage drop equations. 

\subsubsection{Constraints for DGs} We omit $``\forall i\in \mathcal{G}, t\in \mathcal{T}"$ in each constraint on DG for simplicity. Constraints (\ref{logicrelation}) and (\ref{ONOFFBinary}) describe the logic relationship between DG's start-up/shut-down status ($z_{i,t}^u$/$z_{i,t}^d$) and ON/OFF status ($x_{i,t}$). Constraints (\ref{minON}) and (\ref{minOFF}) are the minimum-up and minimum-down time limits. DR chance constraints (\ref{CCgeneUPLIMIT_active}) and (\ref{CCgeneDWLIMIT_active}) ensure that the sum of the actual active power output of DG $i$ and the PFR reserves ($R_{i,t}^{G,U}$/$R_{i,t}^{G,D}$) is within a specified range with a probability at least $1-\epsilon_G$ for all distributions in the ambiguity set $\mathcal{P}_t$. 
Constraints (\ref{DG_UP_FR_reserve_limit}) and (\ref{DG_DW_FR_reserve_limit}) impose limits on upward and downward PFR reserve capacities. Constraints (\ref{CCgeneUPLIMIT_reactive}) and (\ref{CCgeneDWLIMIT_reactive}) are the DRCC for the actual reactive power output bounds of DG $i$ and constraints (\ref{CCRampingUP_active}) and (\ref{CCRampingDOWN_active}) impose probabilistic ramping limits of DG $i$. 

\vspace{-0.2cm}

{\small
\begin{subequations}
\label{generatorsUC}
\begin{align}
\label{logicrelation}
&z_{i,t}^u-z_{i,t}^d=x_{i,t}-x_{i,t-1}\\
\label{ONOFFBinary}
&z_{i,t}^u, z_{i,t}^d, x_{i,t} \in \{0,1\}\\
\label{minON}
&\sum_{\bar{t}=\max\{1,t-T_i^{on}+1\}}^t z_{i,\bar{t}}^u \leq x_{i,\bar{t}}\\
\label{minOFF}
&\sum_{\bar{t}=\max\{1,t-T_i^{off}+1\}}^t z_{i,\bar{t}}^d \leq 1-x_{i,\bar{t}}\\
\label{CCgeneUPLIMIT_active}
&\inf_{\mathbb{P}\in\mathcal{P}_t}\mathbb{P}\left\{g_{i,t}^P-(\alpha_{i,t}^P)^{\text T}\xi_t +R_{i,t}^{G,U} \leq g_{i}^{P,\max}x_{i,t}\right\}\geq 1-\epsilon_{G}\\
\label{CCgeneDWLIMIT_active}
&\inf_{\mathbb{P}\in\mathcal{P}_t}\mathbb{P}\left\{g_{i,t}^P-(\alpha_{i,t}^P)^{\text T}\xi_t -R_{i,t}^{G,D} \geq g_{i}^{P,\min}x_{i,t}\right\}\geq 1-\epsilon_{G}\\
\label{DG_UP_FR_reserve_limit}
&0\leq R_{i,t}^{G,U}\leq R_{i,t}^{G,U,\max}u_t\\
\label{DG_DW_FR_reserve_limit}
&0\leq R_{i,t}^{G,D}\leq R_{i,t}^{G,D,\max}(1-u_t)\\
\label{CCgeneUPLIMIT_reactive}
&\inf_{\mathbb{P}\in\mathcal{P}_t}\mathbb{P}\left\{g_{i,t}^Q-(\alpha_{i,t}^Q)^{\text T}\xi_t \leq g_{i}^{Q,\max}x_{i,t}\right\}\geq 1-\epsilon_{G}\\
\label{CCgeneDWLIMIT_reactive}
&\inf_{\mathbb{P}\in\mathcal{P}_t}\mathbb{P}\left\{g_{i,t}^Q-(\alpha_{i,t}^Q)^{\text T}\xi_t \geq g_{i}^{Q,\min}x_{i,t}\right\}\geq 1-\epsilon_{G}\\
\label{CCRampingUP_active}
&\inf_{\mathbb{P}\in\mathcal{P}_t}\mathbb{P}\left\{\tilde{g}_{i,t}^P-\tilde{g}_{i,t-1}^P \leq r^U_i x_{i,t-1}+S_i^U z_{i,t}^u\right\}\geq 1-\epsilon_{G}\\
\label{CCRampingDOWN_active}
&\inf_{\mathbb{P}\in\mathcal{P}_t}\mathbb{P}\left\{\tilde{g}_{i,t-1}^P-\tilde{g}_{i,t}^P \leq r^D_i x_{i,t}+S_i^D z_{i,t}^d\right\}\geq 1-\epsilon_{G}.
\end{align}
\end{subequations}}

\subsubsection{Constraints for BESSs} 
For a BESS, constraints (\ref{CC_charging}) and (\ref{CC_discharging}) ensure that the sum of the charging and discharging power, upward/downward IR reserves ($IR_{b,t}^{B,U}$/$IR_{b,t}^{B,D}$) and PFR reserves ($R_{b,t}^{B,U}$/$R_{b,t}^{B,D}$) is within the limits with a probability at least $1-\epsilon_B$, for any distribution within $\mathcal{P}_t$. Constraint (\ref{Charge_discharge}) ensures that the charging and discharging behaviors cannot simultaneously happen. Constraint~(\ref{SoC}) models the state of charge level transition dynamics of a BESS. The upward and downward IR reserves are indicated by constraints (\ref{Batteries_up_IR_reserve_limit}) and (\ref{Batteries_down_IR_reserve_limit}). The virtual inertia provided by BESSs is restricted by constraint (\ref{Hbess_limit}). 
Constraints on upward and downward PFR reserves are indicated by constraints (\ref{Batteries_up_FR_reserve_limit}) and (\ref{Batteries_down_FR_reserve_limit}). 
Constraint (\ref{SoCLimit}) imposes limits on the state of charge level in a BESS. Constraint (\ref{SOCFinalequalInitial}) requires that each BESS' state of charge at the final period recovers its initial level. 

{\small
\begin{subequations}
\label{generators}
\begin{align}
\label{CC_charging}
\nonumber &\inf_{\mathbb{P}\in\mathcal{P}_t}\mathbb{P}\left\{p_{b,t}^{ch}+(\beta_{b,t}^{ch})^{\text T}\xi_t +IR_{b,t}^{B,D}+R_{b,t}^{B,D}\leq u_{b,t}^{ch}p_b^{ch,\max}\right\}\\& \qquad \geq 1-\epsilon_{B}\\
\label{CC_discharging}
\nonumber &\inf_{\mathbb{P}\in\mathcal{P}_t}\mathbb{P}\left\{p_{b,t}^{dch}-(\beta_{b,t}^{dch})^{\text T}\xi_t+IR_{b,t}^{B,U}+R_{b,t}^{B,U}\leq u_{b,t}^{dch}p_b^{dch,\max}\right\}\\& \qquad \geq 1-\epsilon_{B}\\
\label{Charge_discharge}
&u_{b,t}^{ch}+u_{b,t}^{dch}=1, u_{b,t}^{ch}/ u_{b,t}^{dch} \in \left\{0,1\right\}\\
\label{SoC}
&E_{b,t}=E_{b,t-1}+\eta^{ch}_bp_{b,t}^{ch}-p_{b,t}^{dch}/\eta^{dch}_b\\
\label{Batteries_up_IR_reserve_limit}
&IR_{b,t}^{B,U}=2H_{b,t}^Bp_b^{dch,\max}RoCoF^{\max}(1-u_t)\\
\label{Batteries_down_IR_reserve_limit}
&IR_{b,t}^{B,D}=2H_{b,t}^Bp_b^{ch,\max}RoCoF^{\max}u_t\\
\label{Hbess_limit}
&H_{b}^{B,\min}\leq H_{b,t}^B\leq H_{b}^{B,\max}\\
\label{Batteries_up_FR_reserve_limit}
&0\leq R_{b,t}^{B,U}\leq Mu_t\\
\label{Batteries_down_FR_reserve_limit}
&0\leq R_{b,t}^{B,D}\leq M(1-u_t)\\
\label{SoCLimit}
&E_b^{\min}\leq E_{b,t}\leq E_b^{\max}\\
\label{SOCFinalequalInitial}
&E_{b,T}=E_{b,0}.
\end{align}
\end{subequations}}

\subsubsection{Constraints for RESs} 
Constraint (\ref{CC_Rene_curt}) ensures the allocated reserve \(\tilde{r}_{s,t}^c\) is sufficient for upward IR/PFR reserves with a probability at least $1-\epsilon_S$ for all distributions within $\mathcal{P}_{s,t}$ of random variable $\xi_{s,t}$. Constraints (\ref{RES_up_IR_reserve_limit}) and (\ref{RES_up_FR_reserve_limit}) connect the upward IR and PFR reserves. The virtual inertia RESs provide is bounded by constraint (\ref{Hres_limit}). Constraint (\ref{Rcurt_Ratio}) imposes limits on the deloading ratio. 
\begin{subequations}
\begin{align}
\label{CC_Rene_curt}
&\inf_{\mathbb{P}\in\mathcal{P}_{s,t}}\mathbb{P}\left\{IR_{s,t}^{S,U}+R_{s,t}^{S,U} \leq \tilde{r}_{s,t}^c\right\}\geq 1-\epsilon_S\\
\label{RES_up_IR_reserve_limit}
&IR_{s,t}^{S,U}=2H_{s,t}^Sp_s^{\max}RoCoF^{\max}(1-u_t)\\
\label{RES_up_FR_reserve_limit}
&0\leq R_{s,t}^{S,U} \leq Mu_t\\
\label{Hres_limit}
&H_{s}^{S,\min}\leq H_{s,t}^S\leq H_{s}^{S,\max}\\
\label{Rcurt_Ratio}
&0\leq\delta_{s,t}\leq \delta_{s}^{\max}
\end{align}
\end{subequations}

\subsubsection{Constraints for Line Power Flows} 
We formulate the uncertain line power flows as a DR quadratic chance constraint:
\begin{align}
\label{CC_Lineflow}
\inf_{\mathbb{P}\in\mathcal{P}_t}\mathbb{P}\left\{(\tilde{f}_{ij,t}^{P})^2 +(\tilde{f}_{ij,t}^{Q})^2\leq (f_{ij}^{\max}) ^2\right\}\geq 1-\epsilon_{L}
\end{align}
where $f_{ij}^{\max}$ is the transmission capacity of line $ij$.

\subsubsection{Constraints for Voltage Magnitudes} 
The following DR chance constraints formulate the upper and lower bounds on the squared voltage magnitudes under uncertainty:
\begin{subequations}
\label{CC_Voltage}
\begin{align}
\label{CC_Voltage_upper}
&\inf_{\mathbb{P}\in\mathcal{P}_t}\mathbb{P}\left\{\tilde{V}_{i,t}\leq (V^{\max})^2 \right\}\geq 1-\epsilon_{V}\\
\label{CC_Voltage_lower}
&\inf_{\mathbb{P}\in\mathcal{P}_t}\mathbb{P}\left\{\tilde{V}_{i,t}\geq (V^{\min})^2 \right\}\geq 1-\epsilon_{V}.
\end{align}
\end{subequations}

\subsubsection{Constraints for Power Exchanges} The transmission capacity of the substation connected to the main grid limits the power exchange between the microgird and the main grid:
\begin{align}
\label{CC_substation}
\inf_{\mathbb{P}\in\mathcal{P}_t}\mathbb{P}\left\{(\tilde{p}_{t}^{ex})^2 +(\tilde{q}_{t}^{ex})^2\leq (S^{\max}) ^2\right\}\geq 1-\epsilon_{E}.
\end{align}

\vspace{-0.4cm}
\subsection{Frequency Constraints involving Uncertainty}
Section \ref{sec.Frequency dynamics} presents the frequency constraints when the power imbalance is realized after unscheduled islanding. When we consider the RES uncertainty, the post-islanding power imbalance \(p_t^{imb} = p_{t}^{ex}-(\gamma_{t}^P)^{\text T}\xi_t\) becomes a random variable. Therefore, we need to embed the uncertainty in those frequency constraints and formulate the following DR chance constraints in order to ensure robustness.

%
\vspace{-0.2cm}
\begin{subequations}
\label{FAIC}
\begin{align}
\label{PowerImbalanceDirectionCCUP}
&\inf_{\mathbb{P}\in\mathcal{P}_t}\mathbb{P}\left\{p_t^{imb}\leq Mu_t \right\}\geq 1-\epsilon_F\\
\label{PowerImbalanceDirectionCCDOWN}
&\inf_{\mathbb{P}\in\mathcal{P}_t}\mathbb{P}\left\{p_t^{imb}\geq-M(1-u_t) \right\}\geq 1-\epsilon_F\\
\label{CC_RoCoFconstraintUP}
&\inf_{\mathbb{P}\in\mathcal{P}_t}\mathbb{P}\left\{p_t^{imb}\leq 2H_t \cdot RoCoF^{\max} \right\} \geq 1-\epsilon_F\\
\label{CC_RoCoFconstraintDOWN}
&\inf_{\mathbb{P}\in\mathcal{P}_t}\mathbb{P}\left\{p_t^{imb}\geq -2H_t \cdot RoCoF^{\max} \right\} \geq 1-\epsilon_F\\
\label{CC_Fnadirconstraint}
&\inf_{\mathbb{P}\in\mathcal{P}_t}\mathbb{P}\left\{\left(p_t^{imb}-z_t\right)^2\leq x_ty_t\right\} \geq 1-\epsilon_F\\
\label{CC_QSSUP}
&\inf_{\mathbb{P}\in\mathcal{P}_t}\mathbb{P}\left\{p_t^{imb}\leq  R^{G,U}_t+ R^{E,U}_t\right\}\geq 1-\epsilon_F\\
\label{CC_QSSDOWN}
&\inf_{\mathbb{P}\in\mathcal{P}_t}\mathbb{P}\left\{p_t^{imb}\geq -R^{G,D}_t-R^{E,D}_t\right\}\geq 1-\epsilon_F.
\end{align}
\end{subequations}

\subsection{Overall Formulation of DR-FCMS}
Before presenting the solution method for DR-FCMS, we summarize the overall DR-FCMS formulation as follows:
\begin{align*}
& \mbox{(DR-FCMS)} \;\; \min \ \eqref{obj}\\
& \text {s.t. \eqref{zt_UP}-\eqref{Nonegative_PFR},  {(\ref{uncer_Renew})-(\ref{Actual_voltage})}, (\ref{Active_Power_Balance_Root_node_nominal})-(\ref{voltage_difference_nominal}), (\ref{Active_Power_Balance_Root_node_real})-(\ref{voltage_difference_real}), (\ref{logicrelation})-(\ref{CCRampingDOWN_active}),}\\
&\text{(\ref{CC_charging})-(\ref{SOCFinalequalInitial}), (\ref{CC_Lineflow}), (\ref{CC_Voltage_upper}), (\ref{CC_Voltage_lower}), (\ref{CC_substation}), and (\ref{PowerImbalanceDirectionCCUP})-(\ref{CC_QSSDOWN})}.
\end{align*}

\section{Solution Methodology} \label{sec.Methods}
Model (DR-FCMS) cannot be readily solved by off-the-shelf solvers due to: 
\begin{enumerate}[label=\arabic*)]
    \item worst-case expected cost terms in~\eqref{obj};
    \item DR chance constraints (\ref{CCgeneUPLIMIT_active}), (\ref{CCgeneDWLIMIT_active}), (\ref{CCgeneUPLIMIT_reactive})-(\ref{CCRampingDOWN_active}), (\ref{CC_charging}), (\ref{CC_discharging}), (\ref{CC_Rene_curt}), (\ref{CC_Lineflow}), (\ref{CC_Voltage_upper}), (\ref{CC_Voltage_lower}), (\ref{CC_substation}), and (\ref{PowerImbalanceDirectionCCUP})-(\ref{CC_QSSDOWN});
    \item bilinear terms in (\ref{Batteries_up_IR_reserve_limit}), (\ref{Batteries_down_IR_reserve_limit}), and (\ref{RES_up_IR_reserve_limit}). 
\end{enumerate}

 In this section, we propose detailed methodologies to address the issues above. We first specify the ambiguity set for describing the RES uncertainty $\xi_t$ and reformulate the worst-case expected cost terms and DR linear chance constraints. There is no straightforward reformulation for the DR quadratic chance constraint~(\ref{CC_Lineflow}), (\ref{CC_substation}), and (\ref{CC_Fnadirconstraint}). Thus, we convert the DR quadratic chance constraints into a set of DR linear chance constraints so that they can be further reformulated to a set of SOC constraints. Subsequently, we linearize the bilinear terms such that we can obtain a computationally tractable mixed-integer convex program for (DR-FCMS).

\vspace{-0.3cm}
\subsection{Wasserstein-metric Ambiguity Set}
A proper selection of ambiguity set for the uncertain forecast error of RES $\xi_t \in \mathbb{R}^{|\mathcal{S}|\times 1}$ is crucial when we reformulate the objective function and DR chance constraints in a realistic and tractable manner. We consider a Wasserstein-metric ambiguity set that defines a ball of radius $\theta \geq 0$ in the distribution space around a prescribed reference distribution:
\begin{align}
\label{Wassambiguityset}
\mathcal{P}_t=\left\{\mathbb{P} \in \mathcal{P}_0: d_W(\mathbb{P},\hat{\mathbb{P}}) \leq \theta \right\}
\end{align}
where $\mathcal{P}_0$ is the set of all probability distributions on $\mathbb{R}^{|\mathcal{S}|}$. The Wasserstein distance 
$d_W(\mathbb{P}_1,\mathbb{P}_2)=\inf_{\mathbb{Q}\in \mathcal{Q}(\mathbb{P}_1,\mathbb{P}_2)}\mathbb{E}_\mathbb{Q}[||\xi_1-\xi_2||]$
where $||\cdot||$ is the general norm on $ \mathbb{R}^{|\mathcal{S}|}$ and $\mathcal{Q}(\mathbb{P}_1,\mathbb{P}_2)$ is the set of all joint distributions on $ \mathbb{R}^{|\mathcal{S}|}\times\mathbb{R}^{|\mathcal{S}|}$ with marginal distributions $\mathbb{P}_1$ and $\mathbb{P}_2$ that govern $\xi_1$ and $\xi_2$, respectively.

Given $N$ historical observations of uncertainty $\xi_t$, i.e., 
$\{\hat{\xi_t^n}\}_{n=1}^N$, a natural choice for $\hat{\mathbb{P}}$ is the empirical distribution $\hat{\mathbb{P}}=\frac{1}{N}\sum_{n=1}^N\delta_{\hat{\xi_t^n}}$, where $\delta_{\hat{\xi_t^n}}$ is the Dirac distribution centered on $\hat{\xi_t^n}$. With this setting, the complexity of the reformulation counterpart for DR chance constraints will rise sharply in the sample size, affecting the scalability \cite{Mohajerin-MP-2018}. To address this, another choice is to use an elliptical distribution as the reference, where the mean and covariance information of the elliptical distribution is estimated from historical observations~\cite{Kuhn-ORMS-2019}. In the literature on power system operations, the Gaussian distribution is widely used to model the forecast errors of RESs \cite{Qi-TSG-2023,ROALD-TPS-2017}, and an elliptical distribution is a generalization of a Gaussian distribution.

Therefore, we consider an elliptical reference distribution denoted as $\hat{\mathbb{P}}=\mathbb{P}_{(\mu_t,\Sigma_t,g)}$ with mean $\mu_t\in \mathbb{R}^{|\mathcal{S}|\times 1}$  and  positive-definite covariance matrix $\Sigma_t\in \mathbb{R}^{|\mathcal{S}|\times|\mathcal{S}|}$ whose probability density function is $f(\xi_t)=k\cdot g\left(\frac{1}{2}(\xi_t-\mu_t)^{\text T}\Sigma_t^{-1}(\xi_t-\mu_t) \right)$ where $k$ is a positive normalization scalar and $g(\cdot)$ is the density generator. 
%
%
As for the Wasserstein distance in $\mathcal{P}_t$, we adopt the Mahalanobis norm associated with the positive definite matrix $\Sigma_t$, defined as $||x||=\sqrt{x^{\text T}\Sigma_t^{-1}x}$. Such choices on reference distribution and Wasserstein distance will lead to the following tractable reformulation and approximation for the worst-case expectation function and DR chance constraints.

\vspace{-0.3cm}
\subsection{Reformulation of Objective Function}
The worst-case expected cost in the objective function (\ref{obj}) can be written in a compact form of $\sup_{\mathbb{P}\in \mathcal{P}_t}\mathbb{E}_\mathbb{P}\left [c(x_t)^{\text T}\xi_t\right]$, where $c(x_t)\in \mathbb{R}^{|\mathcal{S}|\times 1}$ is an affine function on the optimization variable $x_t$ and $\mathcal{P}_t $ is the ambiguity set defined in (\ref{Wassambiguityset}). 

\begin{proposition}
    \(\sup_{\mathbb{P}\in \mathcal{P}_t}\mathbb{E}_\mathbb{P}\left [c(x_t)^{\text T}\xi_t\right]\) with the ambiguity set in (\ref{Wassambiguityset}) and an elliptical reference distribution \(\hat{\mathbb{P}}=\mathbb{P}_{(\mu_t,\Sigma_t,g)}\) is equivalent to:
    \begin{align*}
    \mathbb{E}_{\hat{\mathbb{P}}}\left [c(x_t)^{\text T}\xi_t\right]+\theta||c(x_t)||_{\Sigma_t^{-1}}
    \end{align*}
    where $\hat{\mathbb{P}}=\mathbb{P}_{(\mu_t,\Sigma_t,g)}$ and $||c(x_t)||_{\Sigma_t^{-1}}$ is the Mahalanobis norm associated with inverse matrix $\Sigma_t^{-1}$.
\end{proposition}

Note that Proposition 1 is adapted from Proposition 8 in \cite{Ruan-arxiv-2023} by specifying the form of norm in Wasserstein-metric ambiguity set as the Mahalanobis norm.

\vspace{-0.4cm}
\subsection{Reformulations of DR Linear Chance Constraints}\label{subsec:LinearCC refor}
For ease of exposition, we consider a general single-sided DR linear chance constraint below
\begin{align}
\label{General_LinearCC}
\inf_{\mathbb{P}\in \mathcal{P}_t}\mathbb{P}\left\{a(x_t)^{\text T}\xi_t \leq b(x_t)\right\}\geq 1-\epsilon
\end{align}
where $a(x_t)\in \mathbb{R}^{|\mathcal{S}|\times 1}$ and $b(x_t)\in \mathbb{R}$ are affine functions. The parameter $\epsilon$ is a pre-specified violation probability.

\begin{proposition}
    If $\epsilon < 0.5$, the DR linear chance constraint (\ref{General_LinearCC}) under the ambiguity set defined in (\ref{Wassambiguityset}) can be reformulated as the following SOC constraint:
    \begin{align}
    \label{Reformulation_General_LinearCC}
     \Phi^{-1}(1-\epsilon')\sqrt{a(x_t)^{\text T}\Sigma_t a(x_t)}\leq b(x_t)-a(x_t)^{\text T}\mu_t
    \end{align}
    where $\epsilon'=1-\Phi(\eta^*)$ with $\eta^*$ being the minimal $\eta\geq1-\Phi^{-1}(1-\epsilon)$ that satisfies $\eta(\Phi(\eta)-(1-\epsilon))-\int_{\Phi^{-1}(1-\epsilon)^2/2}^{\eta^2/2}kg(z)dz\geq \theta$.
\end{proposition}

Note that Proposition 2 is an immediate result from Theorem 4.8 of \cite{Chen-MF-2021} and Proposition 2 of \cite{Ruan-arxiv-2023}. See \cite{Ruan-arxiv-2023} and \cite{Chen-MF-2021} for the detailed proof of Proposition 2.

Before applying Proposition 2, we can determine the value of  $\eta^*$ by solving the following optimization problem:
\begin{align}
\nonumber \eta^*=& \min_{\eta \geq 1-\Phi^{-1}(1-\epsilon)} \eta\\
\text{s.t.} \quad & \eta(\Phi(\eta)-(1-\epsilon))-\int_{\Phi^{-1}(1-\epsilon)^2/2}^{\eta^2/2}kg(z)dz \geq \theta.
\end{align}

Let $F(\eta)=\eta(\Phi(\eta)-(1-\epsilon))-\int_{\Phi^{-1}(1-\epsilon)^2/2}^{\eta^2/2}kg(z)dz - \theta.$ Since $F'(\eta)=\Phi(\eta)-(1-\epsilon) \geq 0$ for $\eta \geq 1-\Phi^{-1}(1-\epsilon)$, thereby $F(\eta)$ is monotonically increasing from negative to positive on $\eta \geq 1-\Phi^{-1}(1-\epsilon)$. That is, $\eta^*$ occurs when $F(\eta^*)=0$. We can solve the problem $F(\eta^*)=0$ efficiently via bisection search. 

\vspace{-0.4cm}
\subsection{Reformulations of DR Quadratic Chance Constraints}
{The chance constraints (\ref{CC_Lineflow}), (\ref{CC_substation}), and~\eqref{CC_Fnadirconstraint} contain quadratic terms of random variables, i.e., \(\tilde{f}^P_{ij,t}\), \(\tilde{f}^Q_{ij,t}\), \(\tilde{V}_{i,t}\), \(\tilde{p}_t^{ex}\), \(\tilde{q}_t^{ex}\), and \(p_t^{imb}\). There is no straightforward analytical reformulation for quadratic chance constraints unlike the DR linear ones in Section IV.C \cite{Lubin-arxiv-2016}.}
Here we show how to approximate such DR quadratic chance constraints by using a set of DR linear chance constraints.

\subsubsection{DR Quadratic Chance Constraints on Line Power Flows and Power Exchanges}  Following~\cite{Lubin-arxiv-2016}, we approximate the DR quadratic chance constraint (\ref{CC_Lineflow}) by a set of the DR linear chance constraints and non-probabilistic constraints:
\begin{subequations}
\label{Quadatic_CC_LinePowerFlow}
\begin{align}
\label{Linear_CC_active_LinePowerFlow}
&\inf_{\mathbb{P}\in \mathcal{P}_t }\mathbb{P}\left\{|\tilde{f}_{ij,t}^P|\leq K_{ij,t}^P\right\} \geq 1-\nu_L\epsilon_L\\
\label{Linear_CC_reactive_LinePowerFlow}
&\inf_{\mathbb{P}\in \mathcal{P}_t }\mathbb{P}\left\{|\tilde{f}_{ij,t}^Q|\leq K_{ij,t}^Q\right\} \geq 1-(1-\nu_L)\epsilon_L \\
\label{Quadatic_LinePowerFlow}
&(K_{ij,t}^P)^2+(K_{ij,t}^Q)^2 \leq (f_{ij}^{\max})^2
\end{align}
\end{subequations}
where $K_{ij}^P$ and $K_{ij}^Q$ are the introduced auxiliary variables. The predefined parameter $\nu_L \in (0,1)$ balances the trade-off between violation probabilities in the two DR linear chance constraints (\ref{Linear_CC_active_LinePowerFlow}) and (\ref{Linear_CC_reactive_LinePowerFlow}). DR linear chance constraints (\ref{Linear_CC_active_LinePowerFlow}) and (\ref{Linear_CC_reactive_LinePowerFlow}) involving absolute value operators are essentially two-sided DR linear chance constraints, which can be approximated by two single-sided DR linear chance constraints as:
\begin{subequations}
\label{Quadatic_CC_LinePowerFlow_single}
\begin{align}
\label{Linear_CC_active_LinePowerFlow_single_LEQ}
&\inf_{\mathbb{P}\in \mathcal{P}_t }\mathbb{P}\left\{\tilde{f}_{ij,t}^P\leq K_{ij,t}^P\right\} \geq 1-\nu_L\epsilon_L\\
\label{Linear_CC_active_LinePowerFlow_single_GEQ}
&\inf_{\mathbb{P}\in \mathcal{P}_t}\mathbb{P}\left\{\tilde{f}_{ij,t}^P\geq -K_{ij,t}^P\right\} \geq 1-\nu_L\epsilon_L \\
\label{Linear_CC_reactive_LinePowerFlow_single_LEQ}
&\inf_{\mathbb{P}\in \mathcal{P}_t}\mathbb{P}\left\{\tilde{f}_{ij,t}^Q\leq K_{ij,t}^Q\right\} \geq 1-(1-\nu_L)\epsilon_L\\
\label{Linear_CC_reactive_LinePowerFlow_single_GEQ}
&\inf_{\mathbb{P}\in \mathcal{P}_t}\mathbb{P}\left\{\tilde{f}_{ij,t}^Q\geq-K_{ij,t}^Q\right\} \geq 1-(1-\nu_L)\epsilon_L.
\end{align}
\end{subequations}
We can apply a similar approximation for the DR quadratic chance constraint on power exchange (\ref{CC_substation}).

\subsubsection{DR Quadratic Chance Constraint on Maximum Frequency Deviation} 
By introducing the auxiliary variable $w_t$, We can equivalently transform constraint~\eqref{CC_Fnadirconstraint} as:
\begin{subequations}
\label{CC_Fnadirconstraint_Relaxtion}
\begin{align}
\label{CC_Fnadirconstraint_Linear}
&\inf_{\mathbb{P}\in \mathcal{P}_t}\mathbb{P}\left\{|p_t^{imb}-z_t|\leq w_t\right\} \geq 1-\epsilon_F\\
\label{Fnadirconstraint_Quadratic}
&w_t^2 =x_ty_t\\
\label{Nonnegative_Fnadirconstraint}
&w_t\geq0
\end{align}
\end{subequations}
where (\ref{CC_Fnadirconstraint_Linear}) can be further approximated by two single-sided DR linear chance constraints. 
%
However, equation (\ref{Fnadirconstraint_Quadratic}) is a non-convex quadratic equality constraint. We relax it to the following rotated conic constraint:
\begin{align}
\label{Fnadirconstraint_SOCR}
w_t^2 \leq x_ty_t.
\end{align}

We prove such relaxation is tight, which is also supported by the computational results in Section V-A.
\begin{proposition}
    The constraint~\eqref{CC_Fnadirconstraint} can be replaced by constraints~\eqref{CC_Fnadirconstraint_Linear},~\eqref{Fnadirconstraint_SOCR}, and~\eqref{Nonnegative_Fnadirconstraint} without loss of optimality.
\end{proposition}

\begin{proof}
    Suppose we solve DR-FCMS with constraint~\eqref{CC_Fnadirconstraint}, denoted as \((\hat{M})\), and obtain a partial solution \(\hat{x}, \hat{y}, \hat{z}\), worst-case distribution \(\hat{\mathbb{P}}_t\) and optimal value \(\hat{F}\). On the other hand, suppose we solve DR-FCMS with constraint~\eqref{CC_Fnadirconstraint_Linear},~\eqref{Fnadirconstraint_SOCR}, and~\eqref{Nonnegative_Fnadirconstraint}, denoted as problem \((\tilde{M})\), and obtain a partial solution \(\tilde{w}, \tilde{x}, \tilde{y}, \tilde{z}\), worst-case distribution \(\tilde{\mathbb{P}}_t\) and optimal value \(\tilde{F}\).

    (\(\implies\)) The solution \(\hat{x}, \hat{y}, \hat{z}\) is a feasible solution to model~\((\tilde{M})\) as we can introduce \(\hat{w}_t = \hat{x}_t \hat{y}_t,\ \forall t\) and constraints~\eqref{Fnadirconstraint_Quadratic},~\eqref{Fnadirconstraint_SOCR}, and~\eqref{Nonnegative_Fnadirconstraint} are satisfied. Therefore, \(\hat{F} \geq \tilde{F}\).

    (\(\impliedby\)) By the feasibility of \(\tilde{w}, \tilde{x}, \tilde{y}, \tilde{z}\), we can derive the following inequalities for all \(t\):
    \begin{subequations} \label{eqn:ineqn}
        \begin{align}
            & \quad \inf_{\mathbb{P} \in \mathcal{P}_t} \mathbb{P} \left\{(p_t^{imb}-\tilde{z}_t)^2\leq \tilde{x}_t \tilde{y}_t \right\} \\
            = & \quad \tilde{\mathbb{P}}_t \left\{(p_t^{imb}-\tilde{z}_t)^2\leq \tilde{x}_t \tilde{y}_t \right\} \\
            \geq & \quad  \tilde{\mathbb{P}}_t \left\{|p_t^{imb}-\tilde{z}_t|\leq \tilde{w}_t \right\} \label{eqn:wstep}\\
            \geq & \quad \inf_{\mathbb{P} \in \mathcal{P}_t} \mathbb{P} \left\{|p_t^{imb}-\tilde{z}_t|\leq \tilde{w}_t \right\} \label{eqn:infstep}\\
            \geq & \quad 1 - \epsilon_F. \label{eqn:ccdefinition}
        \end{align}
    \end{subequations}
    Inequality~\eqref{eqn:wstep} is derived from the fact that \(\tilde{w}_t^2 \leq \tilde{x}_t \tilde{y}_t\); inequality~\eqref{eqn:infstep} follows the \(\inf\) operator; inequality~\eqref{eqn:ccdefinition} follows the feasibility of solution~\(\tilde{x},\tilde{y},\tilde{z}\), which satisfies constraint~\eqref{CC_Fnadirconstraint_Linear}. The set of inequalities~\eqref{eqn:ineqn} shows that the solution \(\tilde{x},\tilde{y},\tilde{z}\) also satisfies constraint~\eqref{CC_Fnadirconstraint}. Therefore, \(\tilde{F} \geq \hat{F}\). 

    Combining the results above, we conclude that \(\tilde{F} = \hat{F}\) and it is without loss optimality to replace constraint~\eqref{CC_Fnadirconstraint} by constraints~\eqref{Fnadirconstraint_Quadratic},~\eqref{Fnadirconstraint_SOCR}, and~\eqref{Nonnegative_Fnadirconstraint}.
\end{proof}



So far, the DR quadratic chance constraints in the proposed model (DR-FCMS) have been transformed into single-sided DR linear chance constraints, which can be straightforwardly reformulated as SOC constraints according to Proposition 2.
\vspace{-0.4cm}
\subsection{Linearization of the Bilinear Terms}
Constraints (\ref{Batteries_up_IR_reserve_limit}), (\ref{Batteries_down_IR_reserve_limit}), and (\ref{RES_up_IR_reserve_limit}) involving bilinear terms have the same mathematical form as:
\begin{align}
\label{General_Blinear}
X=\rho Y,
\end{align}
where $X$, $\rho$, and $Y$ are the semicontinuous, binary, and continuous variables. We can introduce the big-M to linearize equation~\eqref{General_Blinear} as follows:
\begin{subequations}
\label{General_BlinearMILP}
\begin{align}
&-M (1- \rho) \leq X-Y \leq M (1- \rho)\\
&-M  \rho \leq X \leq M \rho
\end{align}
\end{subequations}


\vspace{-0.2cm}
\section{Case Studies} \label{sec.Case}
We evaluate our method on a modified IEEE 33-node microgrid system in \cite{Baran-TPD-2016}, adding three DGs, two BESSs, and two RESs. We consider a time horizon of a day with each period spanning an hour (\(|\mathcal{T}| = 24\)) and display the profiles of RES active power forecasts and total loads in Fig.~\ref{fig:RES_load_forecast}. Detailed parameters of the test system are available in \cite{Data-github-2023}. We set additional frequency parameters as $f_0$=50 Hz, $T_{DB}$=0.2 s, $T_E$=1 s, $T_G$=8 s,  $RoCoF^{\max}$=0.5 Hz/s, and $\Delta f^{\max}$= 0.5 Hz. 

To obtain candidate solutions, we generate 100 samples of RES forecast errors from Gaussian distribution whose mean is set to be 0 for all $t\in \mathcal{T}$, and the standard deviation is set to be 5\% of RES forecast values. Later, we generate an additional 10,000 samples from the same Gaussian distribution to test the solution's out-of-sample performance.
We set $\epsilon_G=\epsilon_B=\epsilon_S=\epsilon_L=\epsilon_V=\epsilon_E=\epsilon_F=\epsilon=0.05$ and $\nu_L=\nu_E=0.5$. The radius of the Wasserstein ball is set to \(\theta = 0.01\).

All numerical tests are conducted on a laptop with an Intel Core i7-9750H CPU 2.6 GHz processor and 16 GB memory. All optimization models are coded in MALTAB using YALMIP \cite{Lofberg-2004} and solved by Gurobi v9.0.0~\cite{Gurobi-2022}, with the parameter MIPGap set as 0.1\%.
%
\begin{figure}[t]
	\centering
	\includegraphics[width=0.45\textwidth]{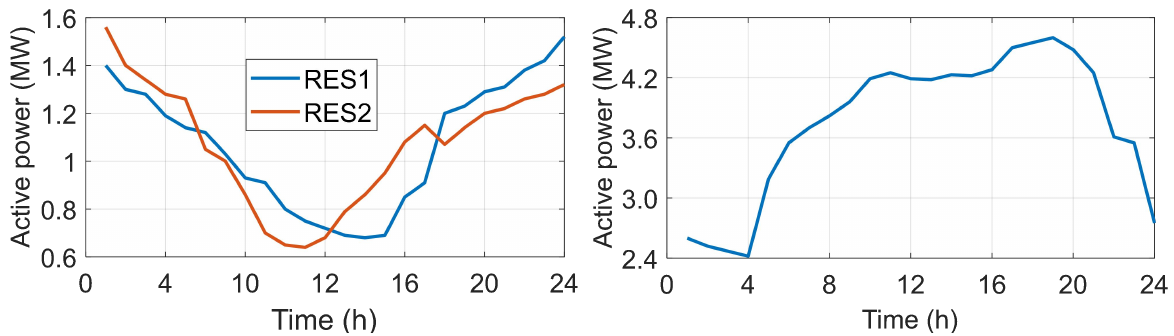}
      \vspace{-0.3cm}
	\caption{Profiles of RES active power forecast values and total load: RESs (left plot) and total load (right plot).}
	\label{fig:RES_load_forecast}
\end{figure}

\vspace{-0.2cm}
\subsection{Tightness of Constraint~\eqref{CC_Fnadirconstraint} Reformulation}
Our reformulation is a relaxation of (DR-FCMS) as we change the equality in constraint~\eqref{Fnadirconstraint_Quadratic} to an inequality~\eqref{Fnadirconstraint_SOCR}. We define the mean relaxation gap as $
\varepsilon_{gap}=\frac{1}{|\mathcal{T}|}\sum_{t\in \mathcal{T}}|x_ty_t-w_t^2|$. We solve the optimization problem ten times and report the maximum, average, and minimum $\varepsilon_{gap}$ in Table \ref{tab:gap}. We observe that the maximum relaxation gap is in the same order as the numerical tolerance, indicating a tight relaxation scheme.
\begin{table}[ht]
\vspace{-0.4cm}
 \centering
 \setlength{\abovecaptionskip}{0.0cm}
 \caption{Values of relaxation gap}
 \vspace{-0.05cm}
 \label{tab:gap}
 \begin{tabular}{ccccccccc}
 \hline\hline
       &Maximum       &Average      &Minimum    \\
 \hline
  $\varepsilon_{gap}$  &1.1e-4     &3.3e-5    &1.7e-7  \\
  \hline\hline
\end{tabular}
\end{table}

\vspace{-0.6cm}
\subsection{The Value of Co-optimization of Power Exchange, PFR Reserves, and Frequency Constraints}
We consider the following three microgrid scheduling methods for comparisons to investigate the value of co-optimization of power exchange, PFR, and frequency constraints.
\begin{itemize}
\item M1: FCMS model in \cite{Chu-TSG-2021}.
\item M2: The microgrid scheduling model with capacity-based PFR constraint in \cite{Wu-TSG-2020}.
\item M3: Our proposed model (DR-FCMS).
\end{itemize}

For a fair comparison, M1 is obtained from M3 by replacing the post-islanding imbalance in (\ref{PowerImbalanceDirectionCCUP})-(\ref{CC_QSSDOWN}) with a constant load increase, and the frequency-related constraints are thus reduced to deterministic counterpart. M2 is obtained from M3 by dropping constraints on RoCoF and maximum frequency deviation (\ref{CC_RoCoFconstraintUP})-(\ref{CC_Fnadirconstraint}). Note that the price for power exchange with the main grid is set less than the cost coefficients of DGs to encourage the power exchange. Fig.~\ref{fig:Pex_PFR} presents the results of power exchange and total PFR reserves produced by M1, M2, and M3, in which the shadowed area is bounded by the maximum and the minimum power exchange realizations among 100 samples of RES uncertainty and the dashed line represents the total PFR reserve. Fig.~\ref{fig:RoCoF_MFD_M1}-\ref{fig:RoCoF_MFD_M3} show the distributions of hourly post-islanding frequency dynamic metrics, i.e., RoCoF and maximum frequency deviation (MFD), which are calculated by $RoCoF=\frac{-p_t^{imb}}{2H_t}$ and (\ref{maximumfrequencydeviation}) with solutions produced by M1, M2, and M3 under 100 in-sample post-islanding realizations.

We observe that the power exchange with the main grid in M1 is the largest and significantly more than those in M2 and M3 since the post-islanding power imbalance considered in M1 decouples the power exchange and is thus not restricted by PFR-related constraints while these work on M2 and M3. As a result, M1 produces the cheapest scheduling scheme (\$1544) against M2 and M3. However, the total PFR reserve in M1 is only upward dedicated to the fixed sudden load increase and notably mismatched with the uncertain power exchange, which makes it hard to ensure frequency security in the event of islanding. As shown in Fig.~\ref{fig:RoCoF_MFD_M1}, the RoCoF and MFD metrics significantly violate the requirements $RoCoF^{\max}=0.5$ Hz/s and $\Delta f^{\max}=0.5 $ Hz in post-islanding realizations.

In contrast, the total PFR reserves in M2 and M3 properly match the power exchange thanks to the co-optimization of power exchange and PFR reserves. Even so, violations of both RoCoF and MFD metrics in M2 exist in many post-islanding realizations due to lacking frequency constraints, as shown in Fig.~\ref{fig:RoCoF_MFD_M2}. Moreover, frequency dynamics are unstable if islanding occurs at hours 1-4, and the corresponding RoCoF and MFD metrics are outside the ranges [-1, 1] Hz/s and [-2, 1] Hz, respectively. By co-optimizing the power exchange, PFR reserves, and frequency constraints, the proposed method (M3) can guarantee frequency security after an unscheduled islanding event in a permissible security level despite its scheduling scheme with a slightly higher cost (\$1795) than M2. As seen in Fig.~\ref{fig:RoCoF_MFD_M3}, the RoCoF metrics for all post-islanding realizations satisfy the requirement $RoCoF^{\max}=0.5$ Hz/s in all hours while there exist three samples at hour 5 and one sample at hour 23 in which the MFD metric is violated due to the DR chance constraint modeling. Still, they satisfy the pre-specified permissible risk level $\epsilon=0.05$.




\begin{figure}[!t]
	\centering
\includegraphics[width=0.4\textwidth]{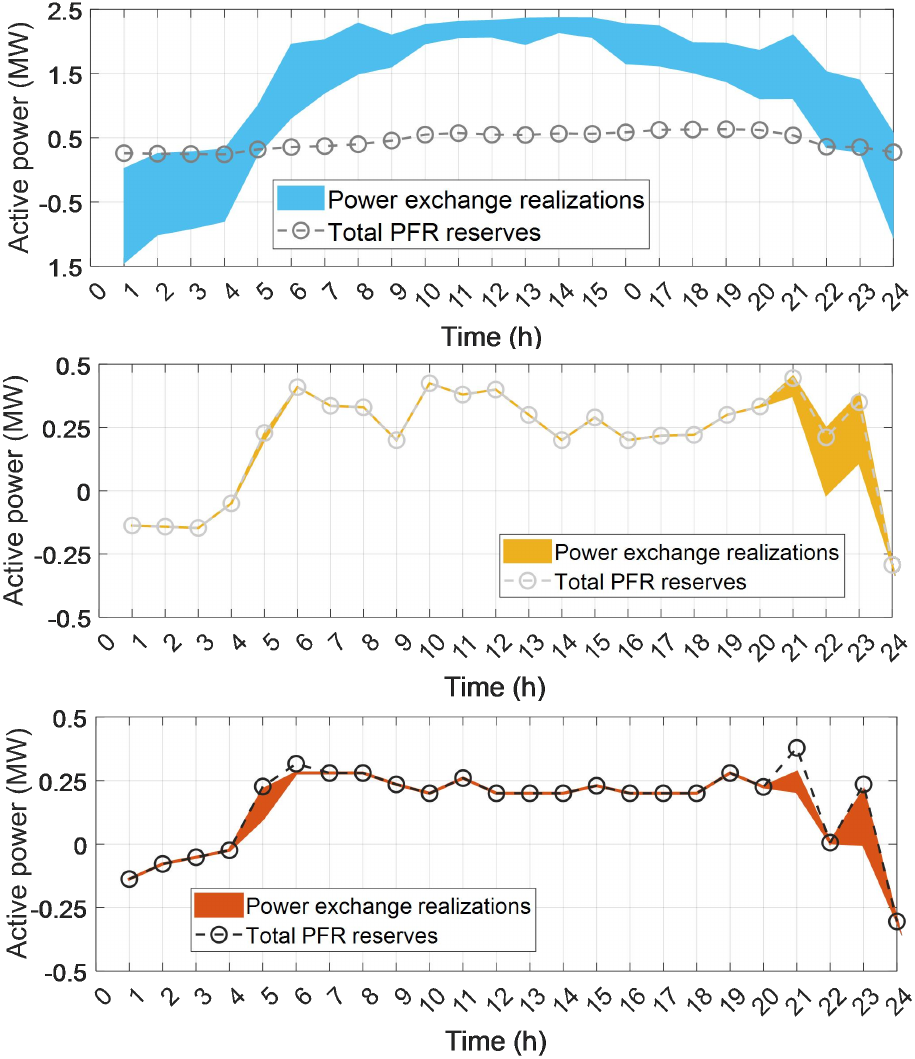}
        \vspace{-0.3cm}
	\caption{Results of power exchange realizations and total PFR reserve: M1 (top), M2 (middle), and M3 (bottom).}
	\label{fig:Pex_PFR}
\end{figure} 
\begin{figure}[!t]
	\centering
	\includegraphics[width=0.4\textwidth]{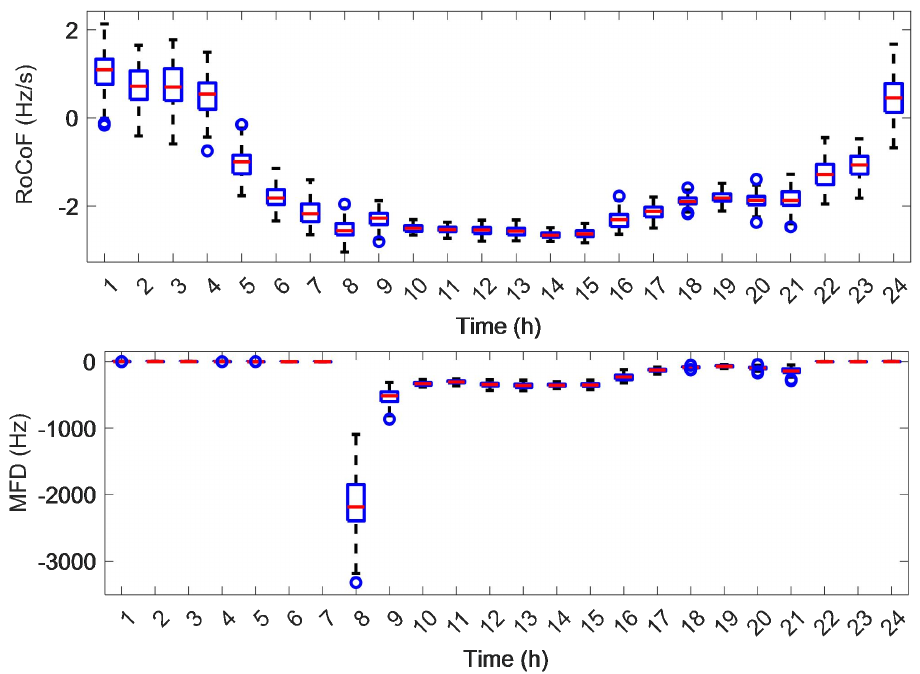}
        \vspace{-0.2cm}
	\caption{The distributions of post-islanding RoCoF (top) and MFD (bottom) in M1 under 100 post-islanding realizations.}
	\label{fig:RoCoF_MFD_M1}
\end{figure} 
\begin{figure}[!t]
	\centering
	\includegraphics[width=0.4\textwidth]{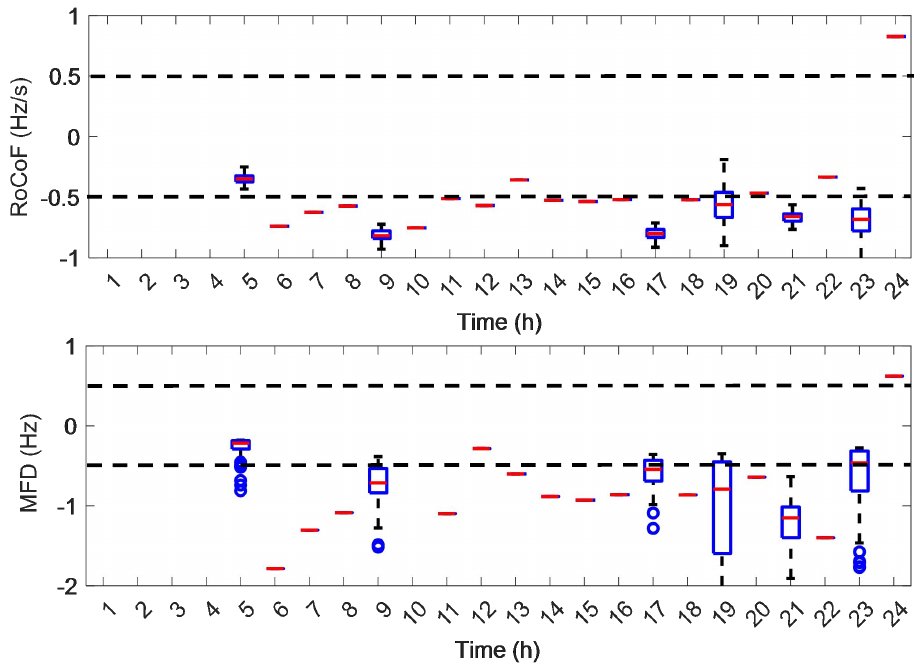}
        \vspace{-0.2cm}
	\caption{The distributions of post-islanding RoCoF (top) and MFD (bottom) in M2 under 100 post-islanding realizations.}
	\label{fig:RoCoF_MFD_M2}
\end{figure} 

\begin{figure}[!t]
	\centering
	\includegraphics[width=0.4\textwidth]{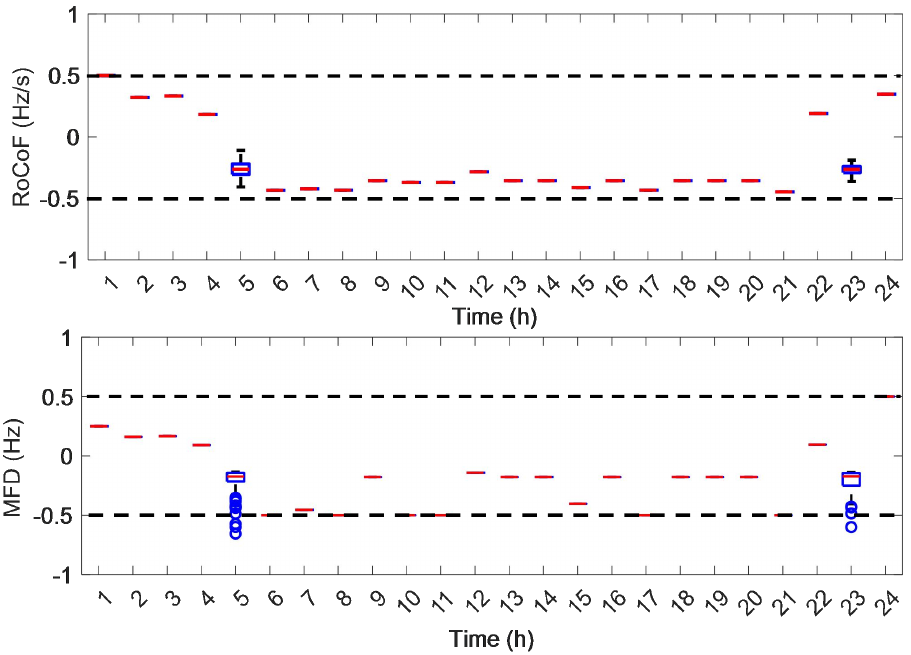}
        \vspace{-0.2cm}
	\caption{The distributions of post-islanding RoCoF (top) and MFD (bottom) in M3 under 100 post-islanding realizations.}
	\label{fig:RoCoF_MFD_M3}
\end{figure} 
\begin{table}[!t]
\vspace{-0.3cm}
 \centering
 \setlength{\abovecaptionskip}{0.0cm}
 \caption{Cost comparisons of benchmark methods}
 \label{tab:cost}
 \begin{tabular}{ccccccccc}
 \hline\hline
 Methods           &M1       &M2      &M3    \\
 \hline
  Total cost (\$)  &1544     &1788    &1795  \\
  \hline\hline
\end{tabular}
\end{table}
\vspace{-0.2cm}
\subsection{Impact of Uncertainty Models on DR-FCMS}
We compare the following four chance-constrained (CC) methods to showcase the value of the adopted Wasserstein-metric DRCC approach for FCMS under uncertainty.

\begin{itemize}
\item Gaussian-CC: Gaussian-assumption CC method.
\item M-DRCC: DRCC with a moment-based ambiguity set.
\item W-DRCC: DRCC with a Wasserstein-metric ambiguity set, built on a uniform reference distribution and the Wasserstein distance defined by an Euclidean norm.
\item EW-DRCC: DRCC with a Wasserstein-metric ambiguity set, built on an elliptical reference distribution and the Wasserstein distance defined by a Mahalanobis norm.
\end{itemize}
Gaussian-CC, M-DRCC, and EW-DRCC are sample-free methods while W-DRCC is a sample-based method whose computation size increases with the number of DR chance constraints and samples.  {Due to the scalability issue, we just apply W-DRCC to reformulate the frequency-related DR chance constraints   (\ref{PowerImbalanceDirectionCCUP})-(\ref{CC_QSSDOWN})  and worst-case expected cost terms in objective function (\ref{obj})} while the remaining DR chance constraints are still addressed by EW-DRCC. For each method, we solve the optimization model ten times and report the maximum, average, and minimum cost, empirical violation probability (EVP), and computation time in Table \ref{tal:comparisons} where we define EVP as the percent of violated samples in the out-of-sample test by averaging over all chance constraints. {Additionally, M-DRCC is infeasible until we increase $\epsilon$ to 0.11 due to its over-conservativeness. Therefore, we include M-DRCC's results with $\epsilon=0.11$ in Table \ref{tal:comparisons} for comparison.}


We observe that {M-DRCC,} W-DRCC, and EW-DRCC yield a higher cost than Gaussian-CC as the DRCC methods hedge against all underlying distributions within the ambiguity sets. Between these DRCC methods, EW-DRCC offers a less costly scheme than M-DRCC and W-DRCC in this case study. As for out-of-sample performance, the EVPs in Gaussian-CC exceed the target value $\epsilon=0.05$ up to 9.55\%, indicating the restriction of the Gaussian distribution assumption. Compared to Gaussian-CC, the EVPs in W-DRCC have been reduced but their average is still slightly over $\epsilon=0.05$ and maximum reaching 7.91\%, {while the EVPs in M-DRCC are significantly less than the predefined $\epsilon=0.11$, indicating its conservativeness again.} Notably, EW-DRCC has nicely maintained low EVPs satisfying the pre-specified requirement $\epsilon=0.05$. Overall speaking, the proposed EW-DRCC can achieve a decent trade-off between cost and EVP.

We further compare the computation time for these methods. As expected, the sample-free methods (Gaussian-CC, EW-DRCC) are significantly faster than the sample-based method (W-DRCC) with around 10 times in terms of average time in this case study despite only partial DR chance constraints being reformulated by W-DRCC. Between Gaussian-CC and EW-DRCC, EW-DRCC is slightly slower than Gaussian-CC but with a comparable complexity as the DR chance constraint is also reformulated as an SOC constraint in EW-DRCC.


\begin{table}[t]
 \centering
 \setlength{\abovecaptionskip}{0.1cm}   
  \caption{Results under four chance-constrained FCMS methods}
  \vspace{-0.1cm}
  \label{tal:comparisons}
  \begin{threeparttable}
  \begin{tabular}{lccccccc}
   \hline\hline
                     &    &Gaussian-CC   &M-DRCC\tnote{*} &W-DRCC     &EW-DRCC    \\
                    
   \hline
                     &max  &1790    &{1809}  &1823      &1799    \\
  Cost (\$)          &avg  &1789    &{1808} &1820      &1796    \\
                     &min  &1788    &{1807}  &1815      &1792    \\
  \hline
                     &max  &9.55    &{3.35}  &7.91      &4.30   \\
  EVP (\%)           &avg  &6.16    &{1.95}  &5.25      &4.02   \\
                     &min  &5.25    &{0.45}  &3.19      &3.82    \\
  \hline
                     &max  &158     &{118}  &932      &143     \\
  Time (s)           &avg  &68      &{80}  &610       &58       \\
                     &min  &15      &{42}  &390      &39       \\
  \hline\hline
\end{tabular}
\begin{tablenotes}
\footnotesize
\item[*] {denotes results of M-DRCC under $\epsilon=0.11$ as it is infeasible until we increase $\epsilon$ to 0.11 in this case study.}
\end{tablenotes}
\end{threeparttable}
\end{table}

\subsection{The Effects of Optimized Deloading Ratios}
We design the following three cases to investigate the effects of optimized deloading ratios on the system operation. As discussed in the technical studies \cite{Ela-Report-2014}, the RES deloading ratio is generally chosen as at most 10\%. We set the fixed deloading ratios for RESs in Case 1 below 10\% as 8\% and set fixed deloading ratios for RESs in Case 2 as 10\%. The optimized deloading ratios for RESs are considered in Case 3 with upper bounds set to 10\%.
\begin{itemize}
\item \emph{Case 1:} Fixed RESs deloading ratios at \(\delta_{s,t} = 8\%\).
\item \emph{Case 2:} Fixed RESs deloading ratios at \(\delta_{s,t} = 10\%\).
\item \emph{Case 3:} Optimized deloading ratios for RESs with upper bounds set to 10\%.
\end{itemize}

As shown in Fig.~\ref{fig:deloading}, we observe that the deloading behaviors for RES1 and RES2 in Case 3 only take place in the periods with high RES and low load, i.e., hours 1, 24 for RES 1 and hours 1-3, 24 for RES2, while in Case 1 and Case 2, deloading behaviors for RES1 and RES2 always exist during the 24-hour dispatch horizon with fixed deloading ratios of 8\% and 10\%, respectively, lacking the co-optimization of deloading ratios, frequency constraints, and power exchange. As a result, Case 3 provides the lowest costly scheme (\$1795), yielding 4.6\% and 5.8\% cost savings compared to Case 1 (\$1881) and Case 2 (\$1905), as illustrated in Table \ref{tal:deloading}.

\begin{figure}[t]
	\centering
	\includegraphics[width=0.42\textwidth]{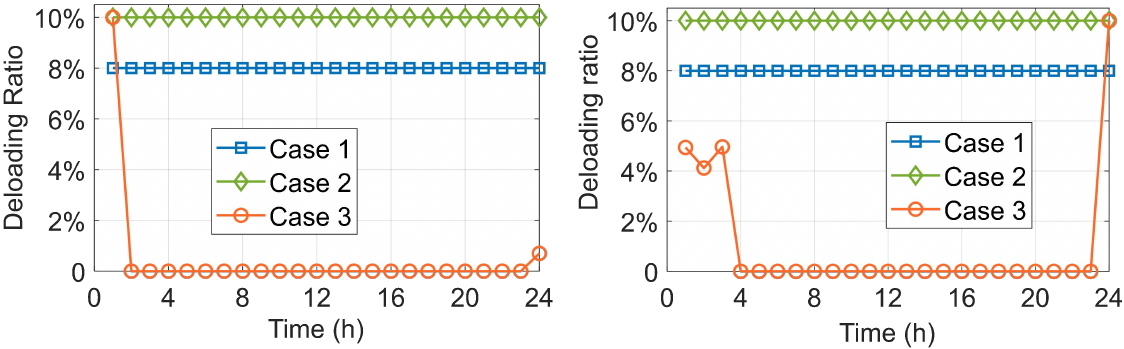}
        \vspace{-0.28cm}
	\caption{Deloading ratios in Case 1, Case 2, and Case 3: RES1 (left plot) and RES2 (right plot)}
	\label{fig:deloading}
\end{figure} 
\begin{table}[!t]
\vspace{-0.3cm}
 \centering
 \setlength{\abovecaptionskip}{0.0cm}
 \caption{Total costs under optimized and fixed deloading ratios}
    \label{tal:deloading}
 \begin{tabular}{ccccccccc}
 \hline\hline
 Methods           &Case 1      &Case 2      &Case 3    \\
 \hline
  Total cost (\$)  &1881     &1905    &1795  \\
  \hline\hline
\end{tabular}
\end{table}

\section{Conclusion} \label{sec.Conclusion}
We propose a DR-FCMS model that co-optimizes unit commitments, power dispatch, power exchange, PFR reserves, and RESs' deloading ratios to achieve seamless unscheduled islanding with frequency security guaranteed. Our model accounts for frequency constraints under unknown and uncertain post-islanding power imbalance by using the DRCC approach, where the corresponding DR quadratic chance constraint is successfully converted into SOC constraints. Consequently, the proposed DR-FCMS model admits a mixed-integer convex programming. We demonstrate the effectiveness of our co-optimization scheme in guaranteeing frequency security following an unscheduled islanding event with numerical results. The proposed DR-FCMS method can reliably limit the EVPs of chance constraints and achieve a better trade-off between cost, EVP, and computation efficiency than commonly used Gaussian-CC, M-DRCC, and W-DRCC in the literature. The numerical results also show the cost saving of optimized deloading ratios for RESs compared to the fixed ones.

\newpage

 



\end{document}